\documentclass[11pt]{article}

\usepackage[english]{babel}
 \usepackage[T1]{fontenc}
\usepackage{fancyhdr}
\usepackage{indentfirst}
\usepackage{graphicx}
\usepackage{newlfont}
\usepackage{cite}
\usepackage{lipsum}

\newcommand\blfootnote[1]{%
  \begingroup
  \renewcommand\thefootnote{}\footnote{#1}%
  \addtocounter{footnote}{-1}%
  \endgroup
}

\newcommand{\mres}{\mathbin{\vrule height 1.6ex depth 0pt width
0.13ex\vrule height 0.13ex depth 0pt width 1.3ex}}

\usepackage{float}
\usepackage{wrapfig}
\usepackage{array}

\usepackage{amsfonts}
\usepackage{amssymb}
\usepackage{amsmath}
\usepackage{latexsym}
\usepackage{amsthm}
\usepackage[margin=1.2in]{geometry}

\DeclareMathOperator{\diam}{diam}
\theoremstyle{plain}                   
\newtheorem{theorem}{Theorem}

\newtheorem{ex}[theorem]{Example}
\theoremstyle{remark}
\newtheorem*{rem}{Remark}
\newcommand{\R}{\mathbb{R}}

\begin{document}
\title {A comparison of Euclidean and Heisenberg Hausdorff measures}
\author{Pertti Mattila and Laura Venieri}

\maketitle 

\begin{abstract} 
We prove some geometric properties of sets in the first Heisenberg group whose Heisenberg Hausdorff dimension is the minimal or maximal possible in relation to their Euclidean one and the corresponding Hausdorff measures are positive and finite. In the first case we show that these sets must be in a sense horizontal and in the second case vertical. We show the sharpness of our results with some examples.

\end{abstract}

\blfootnote{
Both authors are supported by the Academy of Finland through the Finnish Center of Excellence in Analysis and Dynamics Research. L.V. is supported by the Vilho, Yrj\"o ja Kalle V\"ais\"al\"a Foundation.}

\section{Introduction}

Let $\mathcal{H}^s_E$ denote the Euclidean Hausdorff measure in the first Heisenberg group $\mathbb{H}^1$ and let $\mathcal{H}^s_H$ denote the Hausdorff measure with respect to some homogeneous metric. Let $\dim_E$ and $\dim_H$ denote the corresponding Hausdorff dimensions. Generally, for a set $A \subset \mathbb{H}^1$, $\dim_E A$ and $\dim_H A$ can be different. Balogh, Rickly and Serra Cassano in \cite{BRSC} compared them, proving what follows. For $s\geq 0$ let
$$\beta_-(s)=\max\{s,2s-2\}, \beta_+(s)=\min\{2s,s+1\}.$$
Then for any $A\subset\mathbb{H}^1$,
$$\beta_-(\dim_EA) \leq \dim_HA \leq \beta_+(\dim_EA).$$
Moreover, they also showed the sharpness of some of these inequalities, which was then completed by Balogh and Tyson in \cite{BT}: for any $0<s < 3$ they constructed compact subsets $F_1$ and $F_2$ of $\mathbb{H}^1$ such that $\mathcal H^s_E(F_1)$ and $\mathcal H^{\beta_-(s)}_H(F_1)$ are positive and finite and $\mathcal H^s_E(F_2)$ and $\mathcal H^{\beta_+(s)}_H(F_2)$ are positive and finite.  The example $F_1$, for $0<s\leq 2$, is in a sense horizontal and $F_2$ is in a sense vertical. In this paper we show that this must be so. We prove in Theorem \ref{thm1} that for any set $A\subset\mathbb{H}^1$ and $0<s\le2$, if both $\mathcal H^s_E(A)$ and $\mathcal H^{s}_H(A)$ are positive and finite, then in some arbitrarily small neighbourhoods around its typical points $p$, most of $A$ lies close to the horizontal plane through $p$. We shall construct an example (see Example \ref{ex1}) to show that this need not hold for all small neighbourhoods and another example (see Example \ref{ex3}) to show that this does not hold when $s>2$ and both $\mathcal{H}^s_E(A)$ and $\mathcal{H}^{2s-2}_H(A)$ are positive and finite.  Corresponding to the second case we show in Theorems \ref{sgreat1} and \ref{ssmall1} that if both $\mathcal H^s_E(A)$ and $\mathcal H^{\beta_+(s)}_H(A)$ are positive and finite, then in some arbitrarily small neighbourhoods around its typical points $p$, a large part of $A$ lies off the horizontal plane through $p$. 

In \cite{BTW} Balogh, Tyson and Warhurst solved the dimension comparison problem in general Carnot groups, but here we restrict to the first Heisenberg group.

\section{Preliminaries}
In a metric space $X$ for $0 < s<\infty$ the $s$-dimensional Hausdorff measure of $A\subset X$ is defined by
$$\mathcal{H}^s(A)=\lim_{\delta\to 0}\mathcal{H}^s_{\delta}(A),$$
where
$$\mathcal{H}^s_{\delta}(A)=\inf\{\sum_{i=1}^{\infty}\diam(B_i)^s: A\subset\bigcup_{i=1}^{\infty}B_i, \diam (B_i)<\delta\}.$$ 
The Hausdorff dimension of $A$ is
$$\dim A = \inf\{s:\mathcal{H}^s(A)=0\}.$$

Let $B(p,r)$ be the closed ball with centre $p\in X$ and radius $r$. We have the basic upper density theorem for Hausdorff measures, see, e,g., \cite{F}, 2.10.19. 
\begin{theorem}\label{upperdens} Let $A\subset X$ be $\mathcal{H}^s$ measurable with $\mathcal{H}^s(A)<\infty$. Then for $\mathcal{H}^s$ almost all $p\in A$,
\begin{equation*}
2^{-s} \leq \limsup_{r\to 0}\frac{\mathcal{H}^s(A\cap B(p,r))}{(2r)^s} \leq 1
\end{equation*}
and for $\mathcal{H}^s$ almost all $p\in X\setminus A$,
\begin{equation*}
\lim_{r\to 0}\frac{\mathcal{H}^s(A\cap B(p,r))}{(2r)^s} = 0.
\end{equation*}
\end{theorem}

Let $\mathbb{H}^1$ be the first Heisenberg group. It can be identified as $\R^3$ with the non-Abelian group operation
$$p\cdot p' = (x+x',y+y',t+t'+2(x'y-xy')),$$
for $p=(x,y,t), p'=(x',y',t')$, and with the metric 
$$d_H(p,p')=\left(((x-x')^2+(y-y')^2)^2+(t-t'-2(x'y-xy'))^2\right)^{1/4}.$$
In addition to $d_H$ we shall also use the Euclidean metric, which we denote by $d_E$. Then for $0<R<\infty$ there exists a constant $c_R>0$ such that for every $p,p' \in B_E(0,R)$,
\begin{equation}\label{eq2}
\frac{1}{c_R}d_E(p,p') \le  d_H(p,p') \le c_R d_E(p,p')^{1/2}.
\end{equation}
The closed ball $B(p,r)$ is denoted by $B_H(p,r)$ when the metric is $d_H$ and by $B_E(p,r)$ when the metric is $d_E$. The $s$-dimensional Hausdorff measures and dimensions with respect to $d_H$ and $d_E$ are denoted by $\mathcal{H}_H^s, \mathcal{H}_E^s, \dim_H$ and $\dim_E$. In place of $d_H$ we could use any homogeneous metric on $\mathbb{H}^1$, that is, any left invariant metric $d$ satisfying  $d((\delta x, \delta y, \delta^2t),(\delta x', \delta y', \delta^2t'))= \delta d((x, y, t),(x', y', t'))$. By \cite{BLU}, 5.1.5, they are all equivalent.

Recall the definitions of $\beta_-(s)$ and  $\beta_+(s)$ from the introduction. Then by \cite{BTW}, Proposition 3.1, for any positive number $R$ there exists a constant $C_R$ such that for $A\subset B_E(0,R)$ and for $0<s<3$,
\begin{equation}\label{eqcomp}
\mathcal{H}^{\beta_+(s)}_H(A)/C_R\leq\mathcal{H}^s_E(A)\leq C_R\mathcal{H}^{\beta_-(s)}_H(A).
\end{equation}

Let $V(p)$ denote the horizontal plane passing through $p=(x_0,y_0,t_0)\in\mathbb{H}^1$. This is the set of points $q=(x,y,t)$ such that 
\begin{equation}\label{eq4}
t_0-t-2(xy_0-y x_0)=0.
\end{equation}
The Euclidean distance of a point $q=(x,y,t)$ to the plane $V(p)$ is
\begin{equation}\label{eq6}
d_E(q,V(p))= \frac{|t_0-t-2(xy_0-yx_0)|}{\sqrt{1+4(x_0^2+y_0^2)}}.
\end{equation}

We let $A(\delta)$ denote the closed $\delta$ neighbourhood of $A\subset\mathbb{H}^1$ in the Euclidean metric.
Observe that $B_H(p,r)$ looks like $V(p)(r^2)\cap B_E(p,r)$, more precisely, for $p$ as above with $x_0^2+y_0^2 \le R^2$,
\begin{equation}\label{eq5}
V(p)\left(\frac{r^2}{\sqrt{2(1+4R^2)}}\right)\cap B_E\left(p,\frac{r}{2}\right)\subset B_H(p,r) \subset V(p)(r^2)\cap B_E(p,r). 
\end{equation}

The restriction of a measure $\mu$ to a set $A\subset X$ is denoted by $\mu\mres A; \mu\mres A(B)=\mu(A\cap B)$. 

\section{The theorems}\label{The theorems}
 
\begin{theorem}\label{thm1}
Let $0< s \le 2$ and let $A \subset \mathbb{H}^1$ be such that $\mathcal{H}^s_H(A) < \infty$. Then  for $\mathcal{H}^s_E$ almost every $p \in A$ there exists $0<\epsilon<1$ such that
\begin{equation*}
\liminf_{r \rightarrow 0} \frac{\mathcal{H}^s_E(A \cap B_E(p,r) \setminus V(p)(r^{1+\epsilon}))}{r^s}=0.
\end{equation*}
\end{theorem}

\begin{proof} By the Borel regularity of Hausdorff measures we may assume that $A$ is a Borel set. Changing $\epsilon$ a bit it suffices to prove for $\mathcal{H}^s_E$ almost every $p \in A$ that
\begin{equation}\label{claim}
\liminf_{r \rightarrow 0} \frac{\mathcal{H}^s_E(A \cap B_E(p,r) \setminus V(p)(7r^{1+\epsilon}))}{r^s}=0.
\end{equation}
We may assume that for some positive number $R$,
\begin{equation*}
A \subset B_E(0,R).
\end{equation*}
First, let us see that we can reduce to the case when there is a positive number $C$ such that
\begin{equation}\label{red}
\frac{1}{C}\mathcal{H}^s_E(B) \le \mathcal{H}^s_H(B) \le C \mathcal{H}^s_E(B)
\end{equation}
for every $B  \subset A$. The left-hand side inequality holds because of (\ref{eq2}).  
We can decompose $A$ as
\begin{equation*}
A=C \cup D, 
\end{equation*}
where
\begin{equation}\label{AB}
\mathcal{H}^s_E(C)=0 \quad \mbox{and} \quad \mathcal{H}^s_E(B)=0 \Leftrightarrow \mathcal{H}^s_H(B)=0 \quad \forall B \subset D.
\end{equation}
This can be done as follows. Let $\mu_E= \mathcal{H}^s_E \mres {A}$ and $\mu_H= \mathcal{H}^s_H \mres A$. Since $\mu_E << \mu_H$, we have that for every Borel set $B \subset \mathbb{H}^1$,
\begin{equation*}
\mu_E(B)=\int_B D(\mu_E,\mu_H,x) d \mu_H x,
\end{equation*}
where $ D(\mu_E,\mu_H,x)$ is the Radon-Nikodym derivative of $\mu_E$ with respect to $\mu_H$. Thus if we let
\begin{equation*}
C= \{ x \in A: D(\mu_E,\mu_H,x)=0 \}, \quad D=\{x \in A: D(\mu_E,\mu_H,x)>0 \},
\end{equation*}
then $C$ and $D$ satisfy \eqref{AB}. Moreover, we can write 
\begin{equation*}
D= \bigcup_{j=1}^\infty D_j, \quad \mbox{where} \quad \mathcal{H}^s_E(B) \ge \frac{1}{j}\mathcal{H}^s_H(B) \quad \forall B \subset D_j
\end{equation*}
by taking
\begin{equation*}
D_j= \{ x \in D:  D(\mu_E,\mu_H,x)> \frac{1}{j} \}.
\end{equation*}
Thus  \eqref{red} holds for every $D_j$ in place of $A$. If we can prove \eqref{claim} under the assumption \eqref{red}, and so for every $D_j$, it follows that \eqref{claim} holds for $A$ by the the second part of the upper density theorem \ref{upperdens}. Hence we can assume \eqref{red}. 

Let $\epsilon > 0$. Suppose that \eqref{claim} is false. Then by Theorem \ref{upperdens} there exist $c>0$, $0<r_0<1$ and $A' \subset A$, $\mathcal{H}^s_E(A')>0$, such that
\begin{equation}\label{finuppdens}
\mathcal{H}^s_E(A\cap B_E(p,r)) \le 3^sr^s
\end{equation}
and
\begin{equation}\label{ca}
\mathcal{H}^s_E(A \cap B_E(p,r) \setminus V(p)(7 r^{1+\epsilon})) > c r^s
\end{equation}
for every $p \in A'$ and $0<r<r_0$.
Let $0<r<r_0/5$ and $p \in A'$ be such that $r^2<< r^{1+\epsilon}$ and 
\begin{equation}\label{posuppdens}
\mathcal{H}^s_E(B_H(p,r) \cap A') \ge \frac{1}{C} \mathcal{H}^s_H(B_H(p,r) \cap A')   > c'r^s
\end{equation}
with $c'=1/(2C)$ (we can find these by Theorem \ref{upperdens}). Let $k \in \mathbb{N}$ be such that $r^2< r^{(1+\epsilon)^k}$ and $r^2 \ge r^{(1+\epsilon)^{k+1}}$, whence $k \geq \log 2/(2\log(1+\epsilon))$. By the $5r$ covering theorem, see, e.g., Theorem 2.1 in \cite{M}, for $j=1, \dots, k$ we can find $p_{j,i} \in A' \cap B_H(p,r)$ such that
\begin{equation}\label{cover}
A'  \cap B_H(p,r)= \bigcup_{i=1}^{m_j} A'  \cap B_H(p,r) \cap B_E(p_{j,i} ,5  r^{(1+ \epsilon)^j}),
\end{equation}
where the balls $B_E(p_{j,i}, r^{(1+\epsilon)^j})$, $i=1, \dots, m_j$, are disjoint. Since by \eqref{posuppdens}, \eqref{cover} and \eqref{finuppdens}, 
\begin{align*}
c'r^s &< \mathcal{H}^s_E(B_H(p,r) \cap A') \\ &\le \sum_{i=1}^{m_j} \mathcal H^s_E(A'  \cap B_H(p,r) \cap B_E(p_{j,i} ,5 r^{(1+\epsilon)^j}))\\
&\le m_j 3^s 5^s r^{s(1+\epsilon)^j },
\end{align*}
we obtain 
\begin{equation}\label{mj}
m_j \geq c_1r^{s(1-(1+\epsilon)^j)},
\end{equation}
with $c_1=c'/15^s$ depending only on $s$ and $C$.

We can show that the sets
\begin{equation}\label{disj}
\bigcup_{i=1}^{m_j} B_E(p_{j,i},  r^{(1+\epsilon)^j}) \setminus V(p_{j,i})(7 r^{(1+\epsilon)^{j+1}}), \ j=1, \dots,k,
\end{equation}
are disjoint.
Let $j \in \{1, \dots, k-1\}$ and let $B_E(p_{j,i}, r^{(1+\epsilon)^j})$ and $B_E(p_{n,l}, r^{(1+\epsilon)^n})$, $ j +1 \le n \le k$, $i\in \{1, \dots, m_j\}, l \in \{1, \dots, m_n\}$, be such that $B_E(p_{j,i},r^{(1+\epsilon)^j} ) \cap B_E(p_{n,l}, r^{(1+\epsilon)^n}) \neq \emptyset$. We want to show that
\begin{equation}\label{intB}
( B_E(p_{j,i}, r^{(1+\epsilon)^j}) \setminus V(p_{j,i})(7r^{(1+\epsilon)^{j+1}})) \cap B_E(p_{n,l}, r^{(1+\epsilon)^n})= \emptyset.
\end{equation}
Let us denote $p=(\bar{x}, \bar{y}, \bar{t})$, $p_{j,i}=(x_i,y_i,t_i)$, $p_{n,l}=(x_l,y_l,t_l)$. Since $p_{j,i}, p_{n,l} \in B_H(p,r)$, we have
\begin{eqnarray}\label{BH}
 ((\bar{x}-x_i)^2+(\bar{y}-y_i)^2)^2+(\bar{t}-t_i-2(x_i \bar{y}- y_i \bar{x}))^2 \le r^4,
\end{eqnarray}
and 
\begin{eqnarray}\label{BH1}
((\bar{x}-x_l)^2+(\bar{y}-y_l)^2)^2+(\bar{t}-t_l-2(x_l \bar{y}- y_l \bar{x}))^2 \le r^4.
\end{eqnarray}
Moreover, 
\begin{equation}\label{eq3}
d_E(p_{j,i},p_{n,l}) \le r^{(1+\epsilon)^j}+ r^{(1+\epsilon)^n} \le 2 r^{(1+\epsilon)^j}.
\end{equation} 
We now want to show that $d_E(p_{n,l}, V(p_{j,i}) )\le 6 r^{(1+\epsilon)^{j+1}}$. Indeed by \eqref{eq6}, \eqref{BH}, \eqref{BH1} and \eqref{eq3} we have
\begin{align*}
d_E(p_{n,l}, V(p_{j,i}))&= \frac{|t_i-t_l-2(x_l y_i-y_l x_i)|}{\sqrt{1+4(y_i^2+x_i^2)}} \\
 \le& |t_i-t_l-2(x_l y_i-y_l x_i)|\\
 \le& |\bar{t}-t_l-2(x_l \bar{y}-y_l \bar{x})|+|2(y_l \bar{x}-x_l \bar{y})+2(x_i \bar{y}- y_i \bar{x})+2(x_l y_i-y_l x_i)|\\&+|t_i-\bar{t}+2(x_i \bar{y}- y_i \bar{x})|\\
\le & 2r^2+ 2 |(x_l-x_i)(y_i-\bar{y})-(y_l-y_i) (x_i-\bar{x})|\\
=& 2r^2+2| \langle (x_l-x_i,y_i-y_l),(y_i-\bar{y}, x_i- \bar{x}) \rangle |\\
\le & 2r^2+4 r^{(1+\epsilon)^j} r \le 6r^2,
\end{align*}
where $\langle \cdot, \cdot \rangle$ denotes the scalar product and we used Cauchy-Schwarz inequality.
Since $r^2< r^{(1+\epsilon)^k} \le r^{(1+\epsilon)^{j+1}}$, we have $d_E(p_{n,l}, V(p_{j,i})) \le 6 r^{(1+\epsilon)^{j+1}}$. Thus
\begin{equation*}
B_E(p_{n,l}, r^{(1+\epsilon)^n}) \subset V(p_{j,i})(7 r^{(1+\epsilon)^{j+1}}),
\end{equation*}
which implies \eqref{intB}. Hence the sets in \eqref{disj} are disjoint.

We have $\mathcal{H}^s_E(A \cap B_E(p_{j,i},  r^{(1+\epsilon)^j}) \setminus V(p_{j,i})(7  r^{(1+\epsilon)^{j+1}})) > c  r^{s(1+\epsilon)^j}$ by \eqref{ca} hence using \eqref{mj} we get
\begin{eqnarray*}
\begin{split}
\mathcal{H}^s_E(A \cap B_E(p,3r))& \geq \sum_{j=1}^k \sum_{i=1}^{m_j}  \mathcal{H}^s_E(A \cap B_E(p_{j,i},  r^{(1+\epsilon)^j}) \setminus V(p_{j,i})(7  r^{(1+\epsilon)^{j+1}}))\\ &\geq c\sum_{j=1}^k m_j r^{s(1+\epsilon)^j} \geq cc_1\sum_{j=1}^k   r^{s(1-(1+\epsilon)^j)} r^{s(1+\epsilon)^j}\\
&=cc_1kr^s \geq cc_1\log 2/(2\log(1+\epsilon))r^s.
\end{split}
\end{eqnarray*}
When $\epsilon$ is small enough,  the last term is greater than $7^sr^s$. This yields a contradiction with Theorem \ref{upperdens}.
\end{proof}

\begin{rem}
The above proof shows that if $A\subset B_E(0,R)$ satisfies \eqref{red}, then we can choose $\epsilon$ depending only on $s, R$ and $C$.
\end{rem}

\begin{theorem}\label{sgreat1}
Let $s \ge 1$ and $A \subset \mathbb{H}^1$ be such that $\mathcal{H}^{s}_E(A) <\infty$. Then for $\mathcal{H}^{s+1}_H$ almost every $p \in A$ there exists $\delta >0$ such that
\begin{equation}\label{claim2}
\limsup_{r \rightarrow 0} \frac{\mathcal{H}^s_E (A \cap B_E(p,r) \setminus V(p)(\delta r))}{(2r)^s} > \frac{1}{2^{s+1}}.
\end{equation}
\end{theorem}

\begin{proof}
We may assume that $A$ is a Borel set and $A\subset B_E(0,R)$ for some $R>0$. 
We can again reduce to the case where there exists a constant $C>0$ such that for every $B \subset A$ we have
\begin{equation}\label{H12}
\frac{1}{C} \mathcal{H}^s_E(B) \le \mathcal{H}^{s+1}_H(B) \le C  \mathcal{H}^s_E(B).
\end{equation}
Indeed, this follows from a similar reasoning as was used to prove the right-hand side inequality in \eqref{red} since $\mathcal{H}^{s+1}_H <<\mathcal{H}^s_E$ holds always by \eqref{eqcomp} (when $s \ge 1$, $\beta_+(s)=s+1$).

By Theorem \ref{upperdens} for $\mathcal{H}^s_E$ almost all $p \in A$ there exists $0<r_p <1$ such that for every $0<r<r_p$
\begin{equation}\label{H}
 \mathcal{H}^s_E(A \cap B_E(p,r)) \le 3^sr^s
\end{equation}
and
\begin{equation}\label{H2B}
 \mathcal{H}^{s+1}_H(A \cap B_H(p,r)) \le 3^{s+1}r^{s+1}.
\end{equation}
For $j=1, 2, \dots$, let
\begin{equation*}
A_j=\{ p \in A: 2^{-j} \le r_p < 2^{-j+1} \}.
\end{equation*}
Then $\mathcal{H}^s_E (A \setminus \cup_{j=1}^\infty A_j)=0$.

Let $p \in A_l$ and $\mathcal{H}^s_E( A_l )>0$ for some $l$ and let 
\begin{equation}\label{deltap}
0<\delta < \frac{1}{2^{s+2}3^{s+2} C \sqrt{1+4R^2}},
\end{equation}
 where $C$ is as in \eqref{H12}. For every $0<r<2^{-l}, r<\delta,$ we want to show that there exist $p_1, \dots, p_k$, $k \approx (\delta/r) \sqrt{1+4R^2}$, such that
\begin{equation}\label{Vp}
B_E(p,r) \cap V(p)(\delta r) \subset \bigcup_{i=1}^k B_H(p_i, 2r).
\end{equation}
Let $p=(\bar{x}, \bar{y}, \bar{t})$. By \eqref{eq4} the horizontal plane $V(p)$ is the set of points $(x,y,t) \in \mathbb{H}^1$ such that
\begin{equation*}
\bar{t}-t-2(x \bar{y}- y \bar{x})=0.
\end{equation*}
Let $L(p)$ be the vertical line passing through $p$, that is $L(p)=\{(\bar{x}, \bar{y},t): t \in \mathbb{R} \}$. If $q=(\bar{x},\bar{y},t) \in L(p)$ and $d_E(q, V(p))\leq \delta r$ then $|t-\bar{t}| \le \delta r \sqrt{1+4R^2}$. Indeed, by \eqref{eq6}
\begin{align}\label{dEVpL}
\delta r \geq d_E(q, V(p))= \frac{|\bar{t}-t-2(\bar{x} \bar{y}- \bar{y} \bar{x})|}{\sqrt{1+4(\bar{x}^2+\bar{y}^2)}}= \frac{|\bar{t}-t|}{\sqrt{1+4(\bar{x}^2+\bar{y}^2)}}\ge \frac{|\bar{t}-t|}{\sqrt{1+4R^2}}.
\end{align}
Cover the interval $[\bar{t} - \delta r \sqrt{1+4R^2},\bar{t} + \delta r \sqrt{1+4R^2}]$ with intervals  $[t_i,t_{i+1}], i= 1, \dots, k$, with $t_{i+1}-t_i=r^2$ and 
\begin{equation}\label{k}
k\leq 3 (\delta/r)\sqrt{1+4 R^2}.
\end{equation} 
Let
\begin{equation*}
p_i=(\bar{x}, \bar{y}, t_i) \in L(p).
\end{equation*}
If $u \in L(p) \cap V(p)(\delta r)$ then there exists $i$ such that $d_E(u,p_i) \le r^2$ by \eqref{dEVpL}. 
To see that \eqref{Vp} holds, let $q=(x,y,t) \in B_E(p,r) \cap V(p)(\delta r)$ and let $q'$ be the point of intersection between the plane passing through $q$ parallel to $V(p)$ and the line $L(p)$. This means that
\begin{equation*}
q'=(\bar{x},\bar{y}, t') \quad \mbox{and} \quad t-t'-2((\bar{x}-x)\bar{y}-(\bar{y}-y)\bar{x})=0,
\end{equation*}
hence
\begin{equation*}
q'=(\bar{x}, \bar{y}, t-2(y \bar{x}-x \bar{y})).
\end{equation*}
Since $q' \in L(p) \cap V(p)(\delta r)$, there exists $j \in \{1, \dots, k\}$ such that 
\begin{equation}\label{q'pj}
d_E(q',p_j)= |t-2(y \bar{x}-x \bar{y})- t_j| \le r^2.
\end{equation}
Let us now see that $q \in B_H(p_j,2r)$, that is $d_H(q,p_j) \le 2r$. Indeed,
\begin{equation}\label{dHqpj}
d_H(q,p_j)^4= ((x-\bar{x})^2+(y-\bar{y})^2)^2+\left(t-t_j-2(\bar{x}y-\bar{y}x)\right)^2.
\end{equation}
Since $q \in B_E(p,r)$, we have
\begin{equation*}
(x-\bar{x})^2+(y-\bar{y})^2 \le r^2,
\end{equation*}
and by \eqref{q'pj} the second term in \eqref{dHqpj} is $\le r^4$. It follows that  $d_H(q,p_j) \le 2r$, which proves \eqref{Vp}.

Hence by \eqref{Vp}, \eqref{H12}, \eqref{H2B}, \eqref{k} and \eqref{deltap} we have that for every $0<r<2^{-l},$
\begin{align*}
\mathcal{H}^s_E( A \cap B_E(p,r) \cap V(p)(\delta r)) &\le \sum_{i=1}^k \mathcal{H}^s_E(A \cap B_H(p_i,2r)) \\
& \le C \sum_{i=1}^k \mathcal{H}^{s+1}_H(A \cap B_H(p_i,2r)) \\
& \le C  k 3^{s+1}2^{s+1} r^{s+1}\\
& \le  C  3^{s+2}2^{s+1}\frac{\delta}{r} \sqrt{1+4R^2} r^{s+1}  < \frac{1}{2}r^s.
\end{align*}
Thus for $\mathcal{H}^s_E$ almost every $p \in A$ there exists $\delta>0$ such that
\begin{equation*}
\limsup_{r\rightarrow 0} \frac{\mathcal{H}^s_E( A \cap B_E(p,r) \cap V(p)(\delta r)) }{(2r)^s}<\frac{1}{2^{s+1}},
\end{equation*}
which proves \eqref{claim2} by Theorem \ref{upperdens}.
\end{proof}

\begin{rem}
Again, the above proof shows that if $A\subset B_E(0,R)$ satisfies \eqref{H12}, then we can choose $\delta$ depending only on $s, R$ and $C$.
\end{rem}

\begin{theorem}\label{ssmall1}
Let $0 < s<1$ and $A \subset \mathbb{H}^1 $ be such that $\mathcal{H}^s_E(A)<\infty$. Then  for $\mathcal{H}^{2s}_H$ almost every $p \in A$ there exists $\delta>0$ such that
\begin{equation}\label{claim3}
\limsup_{r \rightarrow 0} \frac{\mathcal{H}^s_E (A \cap B_E(p,r) \setminus V(p)(\delta r))}{(2r)^s} > 0.
\end{equation}
\end{theorem}

\begin{proof}
We may assume that $A$ is a Borel set and $A\subset B_E(0,R)$ for some $R>0$. 
Since $ \mathcal{H}^{2s}_H << \mathcal{H}^s_E$ always holds (here $\beta_+(s)=2s$ because $s<1$), we can assume, as in the proof of Theorem \ref{sgreat1}, that there exists $C>0$ such that 
\begin{equation}\label{Hs2s}
\frac{1}{C} \mathcal{H}^s_E(B) \le \mathcal{H}^{2s}_H(B) \le C  \mathcal{H}^s_E(B)
\end{equation}
for every $B\subset A $.

Suppose that \eqref{claim3} does not hold. Let $A_0 \subset A$ be a Borel set such that $\mathcal{H}^s_E(A_0)>0$ and that \eqref{claim3} fails for $p\in A_0$ for every $\delta>0$. Fix $\delta>0$ and $\epsilon>0$, to be chosen sufficiently small at the end of the proof.  Then there exist a Borel set $A' \subset A_0$ and $r_0>0$ such that  $\mathcal{H}^s_E(A')>\mathcal{H}^s_E(A_0)/2$ and for every $p \in A'$ and for every $0<r<r_0$, 
\begin{equation}\label{counterass}
\mathcal{H}^s_E (A \cap B_E(p,r) \setminus V(p)(\delta r)) <\epsilon r^s.
\end{equation}
Let $0<\eta<\min \{r_0,\delta\}$. Let $c=3^s$. Then by Theorem \ref{upperdens}  for $\mathcal{H}^s_E$ almost all $p \in A'$ there is $r_p<\eta$ such that 
\begin{equation}\label{rs}
\frac{r_p^s}{c} \le \mathcal{H}^s_E(A' \cap B_E(p,r_p)) \le c r_p^s.
\end{equation} 
Applying Vitali's covering theorem (see Theorem 2.8 in \cite{M}) to the family of balls $\{ B_E(p,r_p): p \in A' \ \mbox{such that } r_p \mbox{ exists}\}$, we find a subfamily of disjoint balls, $\{B_E(p_i,r_i)\}_{i=1}^\infty$, such that
\begin{equation}
\mathcal{H}^s_E\left( A' \setminus \bigcup_{i=1}^\infty B_E(p_i,r_i) \right)=0.
\end{equation}
Hence we have by \eqref{rs}
\begin{align}\label{HsEAR}
\mathcal{H}^s_E(A') = \mathcal{H}^s_E\left(A'\cap \bigcup_{i=1}^\infty B_E(p_i,r_i)\right)=\sum_{i=1}^\infty  \mathcal{H}^s_E(A' \cap  B_E(p_i,r_i)) \ge  \frac{1}{c} \sum_{i=1}^\infty r_i^s.
\end{align}
Since $p_i \in B_E(0,R)$, we have by \eqref{eq2} that $\mbox{diam}_H(B_E(p_i,r_i))\le c_R \mbox{diam}_E(B_E(p_i,r_i))^{1/2}\le c_R \sqrt{2\eta}$. Let $\eta'= c_R \sqrt{2\eta}$. Then we have
\begin{align}\label{H2sA}
\mathcal{H}^{2s}_{H,\eta'} (A')  \le&   \sum_{i=1}^\infty  \mathcal{H}^{2s}_{H,\eta'}(A' \cap  B_E(p_i,r_i)) \nonumber \\
\le &  \sum_{i=1}^\infty  \mathcal{H}^{2s}_{H,\eta'}(A' \cap  B_E(p_i,r_i) \cap V(p_i)(\delta r_i))\\ &+\sum_{i=1}^\infty  \mathcal{H}^{2s}_{H,\eta'}(A' \cap  B_E(p_i,r_i)\setminus V(p_i)(\delta r_i)).\nonumber
\end{align}
Moreover, we have
\begin{align}\label{diamH}
 \mathcal{H}^{2s}_{H,\eta'}(A'\cap  B_E(p_i,r_i) \cap V(p_i)(\delta r_i)) &\le \mbox{diam}_H(B_E(p_i,r_i) \cap V(p_i)(\delta r_i))^{2s} \nonumber \\ & \le (2( c_R+1) \sqrt{\delta r_i})^{2s}= C'' (\delta r_i)^s. 
\end{align}
To see this, let $q, q' \in B_E(p_i,r_i) \cap V(p_i)(\delta r_i)$ and let $\bar{q}, \bar{q}' \in B_E(p_i,r_i) \cap V(p_i)$ be such that $d_E(q,\bar{q}) \le \delta r_i$ and $d_E(q',\bar{q}') \le \delta r_i$. Then
\begin{align*}
d_H(q,q') \le d_H(q,\bar{q})+d_H(\bar{q},\bar{q}')+d_H(\bar{q}',q').
\end{align*}
We have $ d_H(q,\bar{q}) \le c_Rd_E(q,\bar{q})^{1/2} \le c_R(\delta r_i)^{1/2}$, $ d_H(q',\bar{q}') \le c_R d_E(q',\bar{q}')^{1/2} \le c_R(\delta r_i)^{1/2}$ by \eqref{eq2} and
\begin{equation*}
d_H(\bar{q},\bar{q}') \le d_H(\bar{q},p_i)+d_H(p_i, \bar{q}') = d_E(\bar{q},p_i)+d_E(p_i, \bar{q}') \le 2r_i,
\end{equation*}
where we used the fact that $d_H(u,p_i)=d_E(u,p_i) $ if $u \in V(p_i)$.
Since $r_i < \eta < \delta$, it follows that $r_i \le (\delta r_i)^{1/2}$, hence
\begin{equation*}
d_H(q,q') \le 2 c_R (\delta r_i)^{1/2} +2 r_i \le 2(c_R+1) (\delta r_i)^{1/2},
\end{equation*}
which proves \eqref{diamH}.
On the other hand, by \eqref{Hs2s} and \eqref{counterass} we have
\begin{align*}
\mathcal{H}^{2s}_{H}(A' \cap  B_E(p_i,r_i)\setminus V(p_i)(\delta r_i)) \le C \mathcal{H}^{s}_{E}(A' \cap  B_E(p_i,r_i)\setminus V(p_i)(\delta r_i))<C \epsilon r_i^s,
\end{align*}
thus also
\begin{equation}\label{minus}
\mathcal{H}^{2s}_{H,\eta'}(A' \cap  B_E(p_i,r_i)\setminus V(p_i)(\delta r_i))< C \epsilon r_i^s.
\end{equation}
Hence we have by \eqref{H2sA}, \eqref{diamH}, \eqref{minus} and \eqref{HsEAR}
\begin{align*}
\mathcal{H}^{2s}_{H,\eta'} (A') \le (C''\delta^s+C\epsilon) \sum_{i=1}^\infty r_i^s \le c ( C''\delta^s+C\epsilon) \mathcal{H}^s_E(A'),
\le 2c ( C''\delta^s+C\epsilon) \mathcal{H}^s_E(A_0).
\end{align*}
whence letting $\eta$ and $\eta'$ tend to $0$, 
\begin{align*}
0<\mathcal{H}^{2s}_{H} (A_0) < 2\mathcal{H}^{2s}_{H} (A') \le 2c ( C''\delta^s+C\epsilon) \mathcal{H}^s_E(A_0).
\end{align*}
Since $\delta$ and $\epsilon$ are allowed to depend on $A_0$ and they can be chosen arbitrarily small, we have a contradiction which completes the proof.
\end{proof}

\section{Examples}

We show the sharpness of Theorem \ref{thm1} with three examples.
Example \ref{ex1} shows that we cannot replace $\liminf$ by $\limsup$, Example \ref{ex2} shows that we cannot replace the $r^{1+\epsilon}$-neighbourhood by $Mr^{2}$-neighbourhood for any positive number $M$, in particular we cannot replace it with the Heisenberg ball $B_H(p,r)$. We shall construct these two examples only for $s=1$, but very likely similar examples can be given for any $0<s<2$.  Example \ref{ex3} shows that when $s>2$ then in arbitrarily small neighbourhoods around a point $p$ the set cannot lie too close to the horizontal plane through $p$, in the sense that we cannot obtain the same conclusion as in Theorem \ref{thm1}.

\begin{ex}\label{ex1}
There exists a compact set $F \subset \mathbb{H}^1$ such that for some positive constant $C$, $\mathcal{H}^1_H(F)>0$ and $\mathcal{H}^1_H(A) \leq  C\mathcal{H}^1_E(A)<\infty$ for $A\subset F$, and for $p \in F$, 
\begin{equation}\label{claim4} 
\limsup_{r \rightarrow 0} \frac{\mathcal{H}^1_E (F \cap B_E(p,r) \setminus V(p)(r/8))}{2r} \geq \frac{1}{8}.
\end{equation}
\end{ex}

\begin{ex}\label{ex2}
For any $M, 1<M<\infty$, there exists a compact set $F \subset \mathbb{H}^1$  such that for some positive constant $C$, $\mathcal{H}^1_H(F)>0$ and $\mathcal{H}^1_H(A) \leq  C\mathcal{H}^1_E(A)<\infty$ for $A\subset F$, and for $p \in F$, 
\begin{equation}\label{claim5}
\liminf_{r \rightarrow 0} \frac{\mathcal{H}^1_E (F \cap B_E(p,r) \setminus V(p)(Mr^2))}{2r} \geq \frac{1}{16}.
\end{equation}
\end{ex}

Both examples will follow from the same construction which we now describe. In both cases $F$ will be a subset of the vertical plane $V=\{(x,y,t): y=0\}$, whose points will now be written as $(x,t)$. The metric $d_H$ restricted to this plane is given by
$$d_H((x_1,t_1),(x_2,t_2))=\left((x_1-x_2)^4+(t_1-t_2)^2\right)^{1/4}.$$  For $p=(x,t)\in V$, the horizontal plane $V(p)$ intersects $V$ along the line $\{(u,t): u\in\R\}$. 

For $p,q\in V$ we have $d_E(p,q)\leq d_H(p,q)$ if $d_E(p,q)\leq 1/2$. Thus
$$\mathcal{H}^1_E(B)\leq \mathcal{H}^1_H(B)\ \text{for}\ B\subset V.$$

Let $n$ be an integer, $n\geq 1$, and $\lambda$ a positive number, $0<\lambda \le1/2$. For a rectangle $R=[a,b]\times [c,d]\subset V$ we let $\mathcal R(R,n,\lambda)$ be the collection of the following $2n$ subrectangles:
\begin{align*}
&[a+2i\frac{b-a}{2n},a+2i\frac{b-a}{2n}+\frac{b-a}{2n}]\times [c,c+\lambda (b-a)]\ \text{for}\ i=0,\dots,n-1,\\
&[a+(2i+1)\frac{b-a}{2n},a+(2i+1)\frac{b-a}{2n}+\frac{b-a}{2n}]\times [d-\lambda (b-a),d]\ \text{for}\ i=0,\dots,n-1.
\end{align*}

Let $(n_k)$ be a sequence of integers, $n_k\geq 1$, and $(\lambda_k)$ a sequence of positive numbers, $\lambda_k\leq 1/2$. We define for $k\geq 1$, 

$$\mathcal R_0 = \mathcal R([0,1]^2,1,1/2),$$
$$\mathcal R_{k} = \bigcup_{R\in\mathcal R_{k-1}}\mathcal R(R,n_k,\lambda_k),$$
and 
$$F = \bigcap_{k=0}^{\infty}\bigcup_{R\in\mathcal R_k}R.$$

Then $F\subset V$ is compact and the projection of $F$ on the $x$-axis is $[0,1]$. Thus both $\mathcal{H}^1_E(F)$ and $\mathcal{H}^1_H(F)$ are at least $1$. Using the natural coverings with the rectangles of $\mathcal R_k$, one easily checks that they also are finite provided $\lambda_k$ goes to $0$ sufficiently fast. More precisely, let $h_k$ be the length of the horizontal sides of the rectangles of $\mathcal R_k$ and let $v_k$ be the length of their vertical sides. Then the Euclidean diameter of each $R\in \mathcal R_k$ is $(h_k^2+v_k^2)^{1/2}$ and the Heisenberg diameter is $(h_k^4+v_k^2)^{1/4}$. If $v_k/h_k$ tends to zero as $k\to\infty$, then 
\begin{equation}\label{eq7}
\mathcal{H}^1_E(F\cap R) = h_k\ \text{for}\ R\in\mathcal R_k,
\end{equation}
in particular, $\mathcal{H}^1_E(F)=1$. If moreover, $v_k \le Ch_k^2$ for all large enough $k$, then 
\begin{equation}\label{eq1}
\mathcal{H}^1_H(A)\leq (1+C^2)^{1/4}\mathcal{H}^1_E(A)\ \text{for}\ A\subset F.
\end{equation}

These conditions on $h_k$ and $v_k$  will be satisfied in both examples below; in Example \ref{ex1} $v_k = h_k^2$ and in Example \ref{ex2} $v_k = 34M h_k^2$ for large $k$.

For Example \ref{ex1} we choose $n_k = 2^{^{2^{k-1}-1}}$ and $\lambda_k = 2^{-3\cdot 2^{k-1}}$. As a consequence, the rectangles in $\mathcal R_k$ have horizontal sides of length $h_k=2^{^{-2^k}}$ and the vertical sides of length $2^{^{-2^{k+1}}}$. For $R\in\mathcal R_k$, the horizontal sides of each rectangle $R'$ of $\mathcal R_{k+1}$ inside $R$ thus has the same length as the vertical sides of $R$ (see Figure \ref{fig ex1}). This implies that for $p\in R'$ and $r_k=4h_{k+1}$, 
$B_E(p,r_k) \setminus V(p)(r_k/8)$ contains another rectangle $R''$ of $\mathcal R_{k+1}$. Hence
$$\frac{\mathcal{H}^1_E (F \cap B_E(p,r_k) \setminus V(p)(r_k/8))}{2r_k} \geq \frac{\mathcal{H}^1_E (F \cap R'')}{8h_{k+1}} = \frac{1}{8},$$
from which, recalling also (\ref{eq1}), the asserted properties follow.

\begin{figure}[H]
\begin{center}
\includegraphics[scale=0.3]{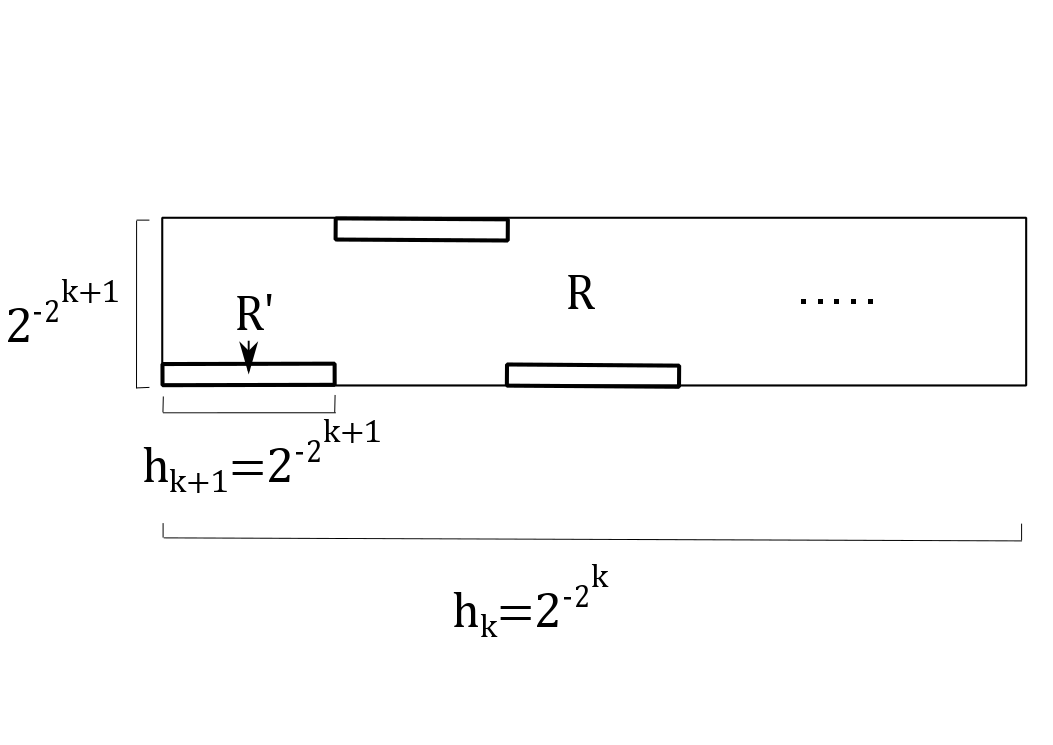}
\caption{A rectangle $R \in \mathcal R_k$ and a rectangle $R' \in \mathcal R_{k+1}$ inside $R$ in Example \ref{ex1}}
\label{fig ex1}
\end{center}
\end{figure}

For Example \ref{ex2} we choose $n_k = 1$ and we let $\lambda_k =1/2$, when $34M4^{-k}\geq 2^{-k}$, that is, $2^k\leq 34M$, and 
$\lambda_k=34M2^{-k-1}$ for all larger $k$. As a consequence, for large enough $k$, the rectangles in $\mathcal R_k$ have horizontal sides of length $h_k=2^{-k}$ and vertical sides of length $34M4^{-k}$. For $R\in\mathcal R_k$, we have two rectangles $R_1$ and $R_2$ of $\mathcal R_{k+1}$ inside $R$, one along the lower side of $R$ and one along the upper. The distance between these rectangles is $34M4^{-k}-2\cdot 34M4^{-k-1}=17M4^{-k}$ (see Figure \ref{fig ex2}).

Let $0<r<1$ and let $k$ be such that $2^{1-k}\leq r < 2^{2-k}$. We assume that $r$ is small enough so that $2^k > 68M$. Let $R, R_1$ and $R_2$ be as above and $p\in R_2$. Then 
$Mr^2 \leq M4^{2-k}< 17M4^{-k}$, whence $R_1$ lies outside $V(p)(Mr^2)$. On the other hand, as $2^{1-k}\leq r$, $R_1$ lies inside $B_E(p,r)$. This implies that 
$B_E(p,r) \setminus V(p)(Mr^2)$ contains $R_1$, from which the asserted properties follow as in the case of Example \ref{ex1}.
\begin{figure}[H]
\begin{center}
\includegraphics[scale=0.25]{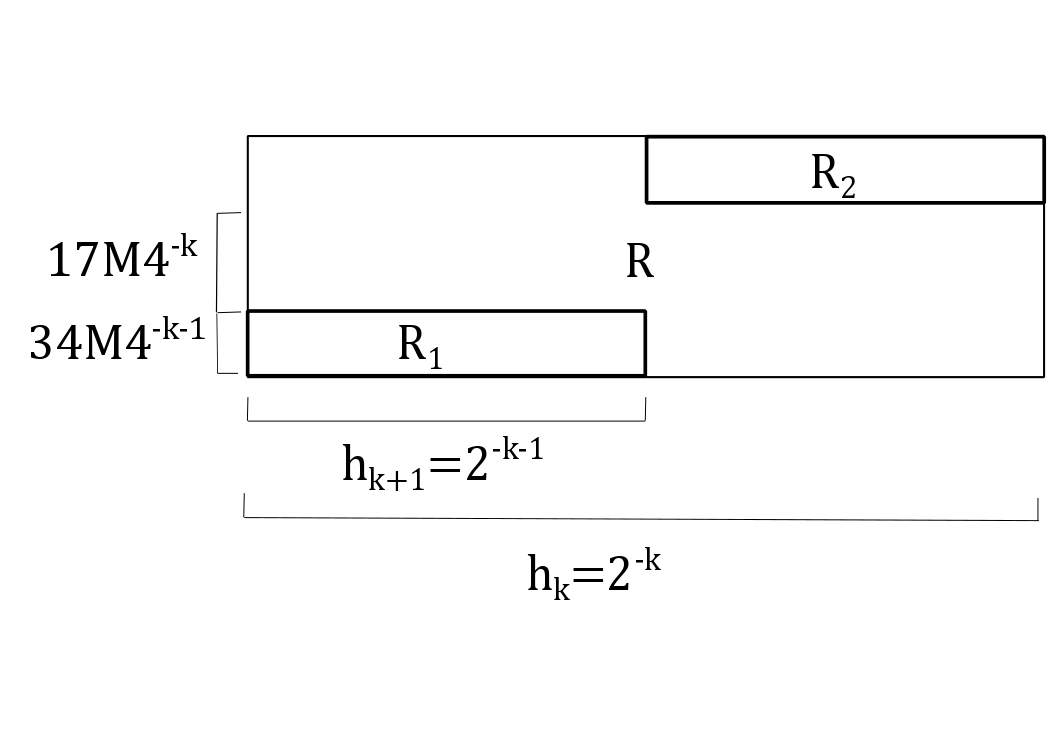}
\caption{A rectangle $R \in \mathcal R_k$ and the rectangles $R_1,R_2 \in \mathcal R_{k+1}$ inside $R$ in Example \ref{ex2}}
\label{fig ex2}
\end{center}
\end{figure}

The next example shows that the conclusion of Theorem \ref{thm1} fails when $2<s<3$.

\begin{ex}\label{ex3}
For any $2<s<3$ there exist constants $c_s, \delta_s>0$ and a set $F_s \subset \mathbb{H}^1$ such that $\mathcal{H}^s_E(F_s)>0$, $\mathcal{H}^{2s-2}_H(F_s)<\infty$ and for $\mathcal{H}^s_E$ almost every $p \in F_s$,
\begin{equation}\label{claimex3}
\liminf_{r\rightarrow 0} \frac{\mathcal{H}^s_E(F_s \cap B_E(p,r)\setminus V(p)( \delta_sr))}{r^s}\ge c_s.
\end{equation}
\end{ex}

This example is taken from Theorem 4.1 in \cite{BT}, where it is used to show the sharpness of some of the dimension inequalities. We will consider the Heisenberg square $Q_H $ and a certain Cantor set above each point of $Q_H$. The Heisenberg square is the invariant set of the affine iterated function system $F_1, F_2,F_3,F_4$, that is $Q_H=\cup_{j=1}^4F_j(Q_H)$. The maps $F_i: \mathbb{H}^1 \rightarrow \mathbb{H}^1$, $i=1,2,3,4$, are similarities with respect to $d_H$ with contraction ratio $1/2$ and they are horizontal lifts of $f_j$, $j=1,2,3, 4$, which are maps in the plane. This means that $\pi \circ F_j=f_j \circ \pi$, where $\pi: \mathbb{R}^3 \rightarrow \mathbb{R}^2$ is the projection $\pi(x,y,t)=(x,y)$. These maps have the form $f_j(x,y)=\frac{1}{2}((x,y)+v_j)$, where $v_1=(0,0)$, $v_2=(1,0)$, $v_3=(0,1)$ and $v_4=(1,1)$. Then we have $\pi(Q_H)=Q=[0,1]^2$, where $Q$ is the invariant set of the iterated function system $f_1, f_2,f_3,f_4$. See \cite{BHT} and  \cite{BT} for more details. We will use the symbolic dynamics notation: for $m \ge 1$ and $w=w_1 w_2 \dots w_m \in W_m= \{1,2,3,4\}^m$ we let $F_w=F_{w_1}\circ \dots \circ F_{w_m}$. Then $Q_H=\cup_{w \in W_m}F_w(Q_H)$ for every $m$. 

Given $2<s<3$, let $d=s-2$ and let $C_d$ be a standard symmetric Cantor set in the $t$-axis such that $0<\mathcal{H}^d_E(C_d)< \infty$. Then $0< \mathcal{H}^{2d}_H(C_d)<\infty$. Moreover, $C_d$ is $d$-Ahlfors regular, which implies that for $0<r<1$ and $(0,0,t') \in C_d$, 
\begin{equation}\label{HdC}
\mathcal{H}^d_E(\{(0,0,t) \in C_d: c_0r \le |t-t'| \le \frac{r}{4}\}) \ge c_d r^d
\end{equation}
for some constants $c_0$ and $c_d$. The set $C_d$  is the invariant set associated to two maps $G_1,G_2$, which are $2^{-1/2d}$-Lipschitz with respect to $d_H$. Let
\begin{equation*}
F_s=\{ (x,y, t+t'): (x,y,t) \in Q_H, (0,0,t') \in C_d\}.
\end{equation*}
It is shown in Theorem 4.1 in \cite{BT} that $\mathcal{H}^s_E(F_s)>0$ and $\mathcal{H}^{2s-2}_H(F		_s)<\infty$. 

Let $p=(\bar{x},\bar{y},\bar{t}+\bar{t}') \in F_s$ and let $0<r<\min\{1/20, c_0/6\}$. Let $m$ be the integer such that $2^{-m+2} \diam_H(Q_H)\le r < 2^{-m+3} \diam_H(Q_H)$, then $2^{-m+2} \diam_H(Q_H) < \min\{1/20, c_0/6\}$. Let $n$ be the smallest integer such that $n \ge 2dm$. For $w \in \{1,2,3,4\}^m$ and $v \in \{1,2\}^n$, let
\begin{equation*}
F^{vw}_s= \{ (x,y,t+t'): (x,y,t) \in F_w(Q_H), (0,0,t') \in G_v(C_d) \} \subset F_s. 
\end{equation*}
Then 
\begin{equation*}
\diam_H(F^{vw}_s ) \le 2^{-m+2} \diam_H(Q_H) \le r.
\end{equation*}
Indeed, if $\diam_H(F^{vw}_s )=d_H((x,y,t+t'),(\tilde{x},\tilde{y},\tilde{t}+\tilde{t}'))$, then we have, as shown in the proof of Theorem 4.1 in \cite{BT},
\begin{align*}
\diam_H(F^{vw}_s )^4&=((x-\tilde{x})^2+(y-\tilde{y})^2)^2+(t+t'-\tilde{t}-\tilde{t}'-2(\tilde{x}y-\tilde{y}x))^2\\
&\le 2 (((x-\tilde{x})^2+(y-\tilde{y})^2)^2+(t-\tilde{t}-2(\tilde{x}y-\tilde{y}x))^2+(t'-\tilde{t}')^2)\\
&\le 2 (2^{-4m} \diam_H(Q_H)^4+2^{-2n/d})\le  2^{-4m+2}\diam_H(Q_H)^4.
\end{align*}

Let $w$ and $v$ be such that $p \in F^{vw}_s$, so we have $F^{vw}_s \subset F_s\cap B_H(p,r)$. Let now $q=(x,y,t+t') \in F^{vw}_s$ be such that $|x-\bar{x}|^2+|y-\bar{y}|^2  \le  r^2/400$. Then $q \in B_H(p,r)$ and we have 
\begin{align}\label{tt'}
\nonumber |t+t'-\bar{t}- \bar{t}'|& \le |t+t'-\bar{t}-\bar{t}'-2(\bar{x}y-\bar{y}x)|+2|\bar{x}y-\bar{y}x|\\ \nonumber & \le d_H(p,q)^2+2|\bar{x}(y-\bar{y})+(\bar{x}-x)\bar{y}| \\ 
 & \le r^2+ \frac{4r}{20}\le \frac{r}{20} + \frac{4r}{20} = \frac{r}{4}.
\end{align}
Let 
\begin{equation}\label{Cdq}
C^q_d=\{(0,0,t'') \in C_d: c_0r \le |t''-t'| \le \frac{r}{4}\}.
\end{equation}
We want to show that the set
\begin{equation*}
L_q=\{(x,y,t+t''): (0,0,t'') \in C_d^q \} 
\end{equation*}
is contained in 
\begin{equation*}
D_r=F_s \cap B_E(p,r) \cap \{(x,y,t): |x-\bar{x}|^2+|y-\bar{y}|^2  \le r^2/400\} \setminus V(p)(c_0r/6).
\end{equation*}
Let $q''=(x,y,t+t'') \in L_q$. Then $q'' \in F_s$ since $(x,y,t) \in Q_H$ and $(0,0,t'') \in C_d$. Moreover, by \eqref{tt'} and \eqref{Cdq} we have
\begin{align*}
|t+t''-\bar{t}-\bar{t}'|\le |t+t'-\bar{t}-\bar{t}'|+|t''-t'| \le \frac{r}{4}+\frac{r}{4}=\frac{ r}{2},
\end{align*}
thus
\begin{align*}
|q''-p|^2=|x-\bar{x}|^2+|y-\bar{y}|^2+|t+t''-\bar{t}-\bar{t}'|^2\le \frac{r^2}{400}+\frac{r^2}{4} < r^2.
\end{align*}
This implies that $q'' \in B_E(p,r)$. It remains to show that $d_E(q'',V(p)) \ge c_0r/6$. Using \eqref{eq6}, \eqref{Cdq} and the facts that $\bar{x}^2+\bar{y}^2\le 2$ and $d_H(p,q) \le r < c_0/6$, we have
\begin{align*}
d_E(q'',V(p))&= \frac{|\bar{t}+\bar{t}'-t-t''-2(x\bar{y}-y\bar{x})|}{\sqrt{1+4(\bar{x}^2+\bar{y}^2})}\\
& \ge \frac{|t'-t''|}{\sqrt{1+4(\bar{x}^2+\bar{y}^2})}-\frac{|\bar{t}+\bar{t}'-t-t'-2(x\bar{y}-y\bar{x})|}{\sqrt{1+4(\bar{x}^2+\bar{y}^2})}\\
& \ge |t'-t''|/3-d_H(q,p)^2\ge c_0r/3-r^2 \ge c_0r/6.
\end{align*}
Hence we have
\begin{equation}\label{Dr}
L_q \subset D_r \subset F_s \cap B_E(p,r)\setminus V(p)(c_0r/6).
\end{equation}
In particular, for every point $(x,y,t) \in F_w(Q_H)$ such that $|x-\bar{x}|^2+|y-\bar{y}|^2\le r^2/400$ there are points $(x,y,t+t'') \in D_r$. Thus
\begin{align*}
\pi(F_s \cap B_E(p,r)\setminus V(p)(c_0r/6))& \supset\pi(D_r)\\&  \supset \pi(F_w(Q_H)) \cap \{(x,y): |x-\bar{x}|^2+|y-\bar{y}|^2\le r^2/400\} \\&=f_w(Q) \cap \{(x,y): |x-\bar{x}|^2+|y-\bar{y}|^2\le r^2/400\} ,
\end{align*}
which implies
\begin{equation}\label{H2pi}
\mathcal{H}^2_E(\pi(D_r)) \ge \mathcal{H}^2_E(f_w(Q)\cap \{(x,y): |x-\bar{x}|^2+|y-\bar{y}|^2\le r^2/400\}) \ge c r^2
\end{equation}
for some constant $c$.
Then by \eqref{Dr}, \eqref{HdC}, \eqref{H2pi} and Theorem 7.7 in \cite{M} we have for some constant $c'>0$, 
\begin{align*}
\mathcal{H}^s_E(F_s \cap B_E(p,r) \setminus V(p)(c_0r/6))&\ge \mathcal{H}^s_E(D_r)\\
& \ge c'\int_{\pi(D_r)} \mathcal{H}^d_E( \{ (0,0,t_q+t''): q=(x_q,y_q,t_q+t'') \in D_r\}) d \mathcal{H}^2_E(x_q,y_q)\\
& \ge c'\int_{\pi(D_r)} \mathcal{H}^d_E (C^q_d) d \mathcal{H}^2_E(x_q,y_q) \\
&\ge c'c_d r^d \mathcal{H}^2_E(\pi(D_r))
\ge c'c_d c r^{2+d}=c_s r^s,
\end{align*}
which implies \eqref{claimex3}.

\vspace{1cm}
\begin{footnotesize}
{\sc Department of Mathematics and Statistics,
P.O. Box 68,  FI-00014 University of Helsinki, Finland,}\\
\emph{E-mail addresses:} 
\verb"pertti.mattila@helsinki.fi", 
\verb"laura.venieri@helsinki.fi" 

\end{footnotesize}

\end{document}